\documentclass[12pt,a4paper]{article}

\usepackage[english]{babel}
\usepackage{amsmath,amssymb,mathtools,mathrsfs,amsthm}
\usepackage[utf8]{inputenc}
\usepackage[T1]{fontenc}
\allowdisplaybreaks

\usepackage{xcolor}
\definecolor{BLUE}{rgb}{0.2, 0.2, 0.6}
\usepackage{hyperref}
\hypersetup{
colorlinks=true,
allcolors=BLUE,
}
\usepackage[authoryear]{natbib}
\setcitestyle{notesep={; },round,aysep={,},yysep={;}}
\bibpunct[,~]{(}{)}{,}{a}{}{;}
\usepackage{graphicx}

\usepackage{bbm}
\usepackage[bottom=2cm,left=2cm,top=2cm,right=2cm]{geometry}

\newtheorem{theorem}{Theorem}

\newtheorem{lemma}[theorem]{Lemma}

\theoremstyle{remark}
\newtheorem{remark}{Remark}
\theoremstyle{definition}
\newtheorem{example}{Example}%[section]

\newcommand{\I}{\mathbbm{1}}
\newcommand{\bs}{\boldsymbol}
\newcommand{\N}{{\mathbb{N}\cup\lbrace0\rbrace}}
\newcommand{\R}{\mathbb{R}}
\newcommand{\Z}{\mathbb{Z}}
\newcommand{\dd}{\mathrm{d}}
\newcommand{\E}{\mathbb{E}}
\newcommand{\ee}{\mathrm{e}}

\newcommand{\Prob}{\mathbb{P}}
\newcommand{\given}{\,\vert\, }

\newcommand{\defin}[1]{\textbf{#1}}

\renewcommand{\ge}{\geqslant}
\renewcommand{\geq}{\geqslant}

\renewcommand{\leq}{\leqslant}

\begin{document}

\title{\sc Product Disintegrations: Some Examples}

\author{%
Luísa~Borsato\footnote{luisabborsato@gmail.com}
\and %% remove this line and below if single author
Eduardo~Horta\footnote{Universidade Federal do Rio Grande do Sul, Brazil. eduardo.horta@ufrgs.br}
\and
Rafael Rigão Souza\footnote{Universidade Federal do Rio Grande do Sul, Brazil. rafars@mat.ufrgs.br}
}

\maketitle

\section{Introduction and general setup}

The main objective of this brief work is to provide a collection of examples involving the concept of a \textbf{product disintegration}, which we introduce here and in a forthcoming paper that addresses the theory and establishes a convergence result (restated here, for convenience, as Theorem~\ref{thm:prop1}, without a proof). Therefore, the present text can be seen as an appendix to the main paper. The concept of a product disintegration generalizes \emph{exchangeability}, a well-known concept in Probability and Statistics \citep{Aldous_1985,Draper1993,Dawid2013,konstantopoulos2019extendibility}, which we recall next.

A sequence $\bs{X} \coloneqq (X_0, X_1,\dots)$ of random variables with values in a Borel space $S$ is said to be \defin{exchangeable} iff for every integer $n\geq1$ and every permutation $\sigma$ of $\{0,\dots,n\}$ it holds that the random vectors $(X_0,\dots,X_n)$ and $(X_{\sigma(0)},\dots, X_{\sigma(n)})$ are equal in distribution. An important characterization of exchangeability is de Finetti's Theorem, which says that a necessary and sufficient condition for a sequence of random variables to be exchangeable is that it is {\it conditionally independent and identically distributed}.
In view of de Finetti's Theorem, we propose to replace the random product measure that characterizes exchangeable sequences---whose factors are all the same---by an arbitrary random product measure, as follows: with $\bs{X}$ as above, we shall say that a sequence $\bs{\xi}\coloneqq\left(\xi_0, \xi_1,\dots\right)$ of random probability measures on $S$ is a \defin{product disintegration of $\bs{X}$} iff, with probability one, the equality
\begin{equation}\label{eq:weaker-product-measure-disintegration}
\Prob\left[X_0\in A_0,\dots,X_n\in A_n\given\bs{\xi}\right] = \xi_0\left(A_0\right)\cdots\xi_n\left(A_n\right)
\end{equation}
holds for each $n\in\N$ and each family $A_0,\dots,A_n$ of measurable subsets of $S$. If $\bs{\xi}$ is a stationary sequence, then we say that $\bs{\xi}$ is a \defin{stationary product disintegration}.
In other words, product disintegration allows one to have conditional independence in terms of a hidden process consisting of a sequence of random probabilities.

\begin{remark}\label{thm:product-disintegration-existence}
Any sequence $\bs{X}\coloneqq (X_0, X_1,\dots)$ of $S$-valued random variables admits a \defin{canonical product disintegration}, obtained by taking $\xi_n(\cdot) = \mathbb{I}_{\lbrace X_n\in\ \cdot\ \rbrace}$ for all integer $n\ge0$.
\end{remark}

Before proceeding, let us fix some notation: in all that follows, $S$ is a compact metrizable space, $C(S)$ is the set of continuous real valued maps defined on $S$, and $M_1(S)$ denotes the set of Borel probability measures on $S$. For $\mu \in M_1(S)$ and $f\in C(S)$, we write $\mu(f) \coloneqq \int_S f(x)\,\mu(\dd x)$. The next result claims that, given a sequence $X_n$, $n\in\N$, of random elements with values in $S$, and a continuous function $f\in C(S)$ (sometimes called an \emph{observable} in the context of Ergodic Theory), a Strong Law of Large Numbers for the process $f\circ X_n$, $n\in\N$, holds if and only if the hidden process used in the product disintegration satisfies an associated Strong Law of Large Numbers itself.

\begin{theorem}\label{thm:prop1}
Let $\bs{X} = \left(X_0, X_1,\dots\right)$ be a sequence of $S$-valued random variables. Assume $\bs{\xi} = (\xi_0, \xi_1,\dots)$ is a product disintegration of $\bs{X}$, and let $f$ be a continuous function from $S$ to $\R$. Then it holds that
\begin{equation}\label{eq:prop1}
\mathbb{P}\left(\lim_{n\to\infty} \frac1n\sum_{i=0}^{n-1}\big(f\circ X_i - \xi_i(f)\big) = 0\right)\ =\ 1.
\end{equation}
In particular, the limit
\begin{math}
X_\infty(f)\coloneqq \lim_{n\to\infty} n^{-1}\sum_{i=0}^{n-1}f\circ X_i
\end{math}
exists almost surely if and only if the limit
\begin{math}
\xi_\infty(f)\coloneqq \lim_{n\to\infty}n^{-1}\sum_{i=0}^{n-1}\xi_i(f)
\end{math}
exists almost surely, in which case one has $X_\infty(f) = \xi_\infty(f)$ almost surely.
\end{theorem}

\section{Examples}\label{sec:examples}

In what follows, given a probability space $(\Omega, \mathscr{F}, \Prob)$, a \defin{random probability measure on $S$} is a Borel measurable mapping $\xi\colon\Omega \rightarrow M_1(S)$. Throughout this section, its value at a point $\omega\in\Omega$ is denoted by $\xi^\omega$ and $\xi(\omega,\cdot)$ interchangeably.

\subsection{Product disintegrations \emph{per se}}
\begin{example}[Product disintegrations are not (necessarily) unique]\label{example:non-unique}
Let $\bs\vartheta=(\vartheta_n:\,n\in\N)$ be a sequence of independent and identically distributed random variables, uniformly distributed in the unit interval $[0,1]$, and let, for $n\in\N$, $\xi_n$ be the random probability measure on $S\coloneqq\lbrace0,1\rbrace$ defined via $\xi_n^\omega(1) \coloneqq \vartheta_n(\omega)$, where for simplicity we write $\xi_n(x)$, $x\in S$, instead of $\xi_n(\{x\})$. Assume further that $\bs\xi$ is a product disintegration of a given sequence $\bs X$ of Bernoulli random variables. Without much effort, it is possible to prove that $\sigma(\bs\xi) = \sigma(\bs\vartheta)$ and, in particular, it holds that, conditionally on $\bs{\xi}$, each $X_n$ is a Bernoulli random variable with parameter $\vartheta_n$. That is, for each $n\in\N$ we have
\(
\Prob(X_n = 1\given \bs{\xi}) = \vartheta_n.
\)
Now define $\hat{\xi}_n\colon\Omega\rightarrow M_1(S)$ by $\hat{\xi}_n^\omega(1) \coloneqq \delta_{X_n(\omega)}(1) = \I_{[X_n=1]}(\omega)$, so that $\hat{\bs{\xi}}\coloneqq (\hat{\xi}_0,\hat{\xi}_1,\dots)$ is the canonical product disintegration of $\bs{X}$. Clearly $\bs{\xi}$ and $\hat{\bs{\xi}}$ are different since $\hat{\xi}_n^\omega$ is equal either to $\delta_{\lbrace0\rbrace}$ or $\delta_{\lbrace1\rbrace},$ whereas this is not true of $\xi_n$. Indeed, for $\theta\in[0,1)$, we have \(\Prob(\xi_n(1)\leq\theta) = \Prob(\vartheta_n\leq\theta) = \theta\),
whereas
\(
\Prob(\hat{\xi}_n(1)\leq \theta) = \Prob(\I_{[X_n=1]}\leq\theta)
= \Prob(\I_{[X_n=1]} = 0)
= \Prob(X_n = 0).
\)
\end{example}

\begin{example}[Random Walk as a two-stage experiment with random jump probabilities]
In the same setting as Example~\ref{example:non-unique}, let $Z_n \coloneqq 2X_n-1$, $n\in\N$. Clearly $\bs Z = (Z_0, Z_1,\dots)$ is an independent and identically distributed sequence of standard Rademacher random variables, i.e., for each $n \in \N$ it holds that $\Prob(Z_n = +1) = \Prob(Z_n = -1) = \frac{1}{2}$. Indeed,
for any $x_0,x_1,\dots,x_n\in \{0,1\}$, we have
\(
\Prob(Z_0 = 2x_0-1, \dots, Z_n = 2x_n-1) =
\Prob\left(X_0 = x_0 , \dots, X_n = x_n \right)
= \E \prod_{j=0}^n \xi_j(x_j)
= \prod_{j=0}^n \E\xi_j(x_j),
\)
where the last equality follows from the assumption that the $\vartheta_n$'s are independent. Moreover, $\E \xi_j(x_j) = 1/2$ since the left-hand side in this equality is either $\E\vartheta_j$ or $1 - \E\vartheta_j$. Now let $S_0\coloneqq 0$ and $S_n = Z_0 + \cdots + Z_{n-1}$ for $n\geq 1$. By the above derivation, $(S_n\colon\,n\geq 0)$ is the symmetric random walk on $\Z$. Therefore, although --- unconditionally --- at each step the process $(S_n)$ jumps up or down with equal probabilities, we have that conditionally on $\bs\xi$ it evolves according to the following rule: at step $n$, sample a Uniform$[0,1]$ random variable $\vartheta_n$ independent of anything that has happened before (and of anything that will happen in the future), and go up with probability $\vartheta_n$, or down with probability $1 - \vartheta_n$.
\end{example}

\begin{example}
Let $\bs{X} = (X_0, X_1, \dots)$ be an exchangeable sequence of Bernoulli$(p$) random variables. In particular, $\bs X$ satisfies
\begin{equation*}\label{eq:productmeasure}
\Prob\left(X_0 = x_0,\dots,X_n = x_n\given\vartheta\right) = \prod_{i=0}^{n} \vartheta^{x_i}(1-\vartheta)^{1-x_i}.
\end{equation*}
for some random variable $\vartheta$ taking values in the unit interval. Then, defining the random measures $\xi_n$ via $\xi_n(\{1\})\coloneqq \vartheta$ for all $n$, it is clear that $(\xi_0,\xi_1,\dots) =:\bs\xi$ is a stationary product disintegration of $\bs{X}$ --- again using the fact that $\sigma(\bs\xi) = \sigma(\bs\vartheta)$. In particular, in this scenario, an unconditional Strong Law of Large Numbers \emph{does not hold} for $\bs{X}$, unless when $\vartheta$ is a constant. See also Theorem~2.2 in \cite{taylor1987laws}, which provides a characterization of the Strong Law for the class of integrable, exchangeable sequences. This example illustrates the fact that the existence of a product disintegration is not sufficient for the Law of Large Numbers to hold (indeed, by Remark~\ref{thm:product-disintegration-existence}, any sequence of random variables admits a product disintegration).
\end{example}

\begin{example}[Concentration inequalities]\label{example:concentration}
One important consequence of the notion of a product disintegration is that it allows us to easily translate certain concentration inequalities (such as the Chernoff bound, Hoeffding's inequality, Bernstein's inequality, etc) from the independent case to a more general setting. Recall that the classical Hoeffding inequality says that,
%taken from Boucheron et al. p. 34-35
if $\bs X = (X_0, X_1,\dots)$ is a sequence of independent random variables with values in $\left[0, 1\right]$, then one has the bound
\(
\Prob\left(S_n \geq t\right)\leq\exp\left(-{2t^2}/{n}\right)
\)
for all $t>\E S_n$, where $S_n\coloneqq\sum_{i=0}^{n-1} X_i$.

\begin{theorem}[Hoeffding-type inequality]\label{thm:hoeffding-new}
Let $\bs X = \left(X_0, X_1,\dots\right)$ be a sequence of random variables with values in the unit interval $S := \left[0,1\right]$, and let $\bs\xi = (\xi_0,\xi_1,\dots)$ be a product disintegration of $\bs X$. Then, for any $t>0$, it holds that
\begin{equation*}
\Prob\left(S_n \geq t\given\E(S_n\given\bs\xi) < t\right) \leq\exp\left(-{2t^2}/{n}\right),
\end{equation*}
where $S_n\coloneqq\sum_{i=0}^{n-1} X_i$.
\end{theorem}

\begin{proof} By applying the classical Hoeffding inequality to the probability measures $\Prob(\cdot\given\bs\xi)_\omega$, we have
\begin{equation*}
\Prob\left( S_n \geq t\given\bs{\xi}\right)\I_{\lbrace\E(S_n \given\bs\xi) < t\rbrace}
\leq \exp\left(-{2t^2}/{n}\right)\I_{\lbrace\E(S_n \given\bs\xi) < t\rbrace}.
\end{equation*}
Taking the expectation on both sides of the above inequality, and dividing by $\Prob(\E(S_n\,\given\,\bs\xi)< t)$, yields the stated result.
\end{proof}

Notice that if $\bs\xi$ is the canonical product disintegration of $\bs X$, then the above theorem is not very useful: indeed in this case we have $\E(S_n\given\bs\xi) = S_n$, so the left-hand side in the inequality is zero. The above theorem also tells us that, for $t>0$,
\begin{align*}
\Prob\big(S_n \geq t\big) &= \Prob\big(S_n \geq t\,\big\vert\,\E(S_n\given\bs\xi) < t\big)\times\Prob\big(\E(S_n\given\bs\xi) < t\big)\\
&\qquad\qquad +\quad\Prob\big(S_n \geq t\,\big\vert\,\E(S_n\given\bs\xi) \geq t\big)\times\Prob\big(\E(S_n\given\bs\xi) \geq t\big)\\
&\qquad\qquad\qquad\leq\quad \exp\left(-\frac{2t^2}{n}\right) \,+\, \Prob\big(\E(S_n\given\bs\xi) \geq t\big),
\end{align*}
so the rate at which $\Prob\big(S_n \geq t\big)\to0$ as $t\to\infty$ is governed by the rate at which $\Prob\left(\E(S_n\,\big\vert\,\bs\xi) \geq t\right)\to\infty$ as $t\to\infty$. To illustrate, let us consider two extreme scenarios, one in which $\xi_n = \xi_0$ for all $n$ (so that $\bs X$ is exchangeable) and one in which the $\xi_n$'s are all mutually independent: in the first case, we have that $\E\big(S_n\given\bs\xi\big) = n \int_0^1 x\,\xi_0(\dd x)$, and thus the rate at which $\Prob\big(S_n \geq t\big)\to0$ as $t\to\infty$ depends only on the distribution of the random variable $\int_0
^1 x\,\xi_0(\dd x)$. On the other hand, if the $\xi_n$'s are independent, then we have that $\E\big(S_n\given\bs\xi\big) = \sum_{i=0}^{n-1}\int_0^1 x\,\xi_n(\dd x)$, and in this case the summands are independent random variables with values in the unit interval. Therefore, we can apply the classical Hoeffding inequality to these random variables to obtain the upper bound $\Prob(S_n\geq t)\leq 2\exp(-2t
^2/n)$ for $t>\E S_n$ (in fact, we already know that the upper bound $\exp(-2t^2/n)$ holds, since independence of the $\xi_n$'s entails independence of the $X_n$'s).
\end{example}

\begin{example}
Let $S\coloneqq [a,b]^d$ where $d$ is a positive integer and $a<b\in\R$. Given a sequence $\bs{X} = (X_1, X_2, \dots)$ of $S$-valued random variables, we shall write $X_n = (X_n^{1},\dots,X_n^{d})$. Suppose $\bs\xi = (\xi_0,\xi_1,\dots)$ is a product disintegration of $\bs X$. Equation \eqref{eq:weaker-product-measure-disintegration} then yields, for all measurable sets $A_i^j\subseteq [a,b]$, with $i\in\lbrace0,\dots,n\rbrace$ and $j\in\lbrace1,\dots,d\rbrace$, the equality
\begin{align*}
\Prob(X_0^1\in A_0^1,&\dots X_0^d\in A_0^d,\dots, X_n^1\in A_n^1,\dots,X_n^d\in A_n^d\given\bs{\xi})\\
&=\xi_0(A_0^1\times\cdots\times A_0^d)\cdots\xi_n(A_n^1\times\cdots\times A_n^d).
\end{align*}
An identity as above appears naturally in statistical applications, for instance when one observes samples of size $d$, $(X_n^1,\dots,X_n^d)$, $n=0,1,\dots$, from distinct ``populations'' $\xi_0, \xi_1,\dots$ --- we refer the reader to \cite{petersen2016functional} and references therein for details.
\end{example}

\subsection{Convergence}
\begin{example}[Regime switching models]\label{exmp:regime-switch}
Let $S = \{-1,1\}$ and put $M'\coloneqq \{\mu,\lambda\}\subseteq M_1(S)$ with $\mu(1)>\lambda(1)$. The measures $\mu$ and $\lambda$ are to be interpreted as $2$ distinct ``{regimes}'' (for example, expansion and contraction, in which case one would likely assume $\mu(1)>1/2>\lambda(1)$). Let $(Q_{ij}\colon\,i,j\in M')$ be a row stochastic matrix with stationary distribution $\pi = (\pi_\mu, \pi_\lambda)$. Let $\bs{\xi}\coloneqq (\xi_0,\xi_1,\dots)$ be a Markov chain with state space $M'$, initial distribution $\pi$ and transition probabilities $(Q_{ij})$. Notice that $\E\xi_n = \mu \pi_\mu + \lambda \pi_\lambda$ for all $n$.

Assume $\bs{X}\coloneqq (X_0, X_1,\dots)$ is a sequence of $S$-valued random variables and that $\bs\xi$ is a product disintegration of $\bs X$. Then we have, for $x\in \{-1,1\}$, that $\Prob(X_n = x) = \E\xi_n(x) = \mu(x)\pi_\mu + \lambda(x)\pi_\lambda$. We also have, for $x_0,x_1\in \{-1,1\}$,
\begin{align}\label{eq:markov-switching}
\begin{split}
\Prob(X_0 = x_0, X_1 = x_1) &= \E \xi_0(x_0)\xi_1(x_1)\\
&= \mu(x_0)\mu(x_1)\pi_{\mu}Q_{\mu\mu} + \mu(x_0)\lambda(x_1)\pi_{\mu}Q_{\mu\lambda}\\
&\qquad +\lambda(x_0)\mu(x_1)\pi_\lambda Q_{\lambda\mu} + \lambda(x_0)\lambda(x_1)\pi_\lambda Q_{\lambda\lambda}.
\end{split}
\end{align}
This shows that in general it may be difficult to compute the finite-dimensional distributions of the process $(X_0, X_1,\dots)$ --- although this process inherits stationarity from $\bs\xi$. Also, an easy check tells us that generally speaking $\bs X$ is not a Markov chain.

Nevertheless, assuming $Q$ is irreducible and positive recurrent (i.e., $\pi_\mu,\pi_\lambda\notin\lbrace0,1\rbrace$), we have by the ergodic theorem for Markov chains, that
\begin{equation}\label{eq:markov-chain-ergodic-theorem}
\lim_{n\to\infty}\frac{1}{n}\sum_{k=0}^{n-1}h\circ\xi_k = \pi_\mu h(\mu) + \pi_\lambda h(\lambda) = \E h\circ\xi_0,\qquad{\mathrm{a.s}},
\end{equation}
for any bounded $h\colon M'\to\R$.
Now let $f\colon S\to \R$ and consider the particular case where $h\circ \xi := \xi (f)$. Equation \ref{eq:markov-chain-ergodic-theorem} becomes
\begin{equation}\label{eq:markov-chain-ergodic-theorem-nova}
\lim_{n\to\infty}\frac{1}{n}\sum_{k=0}^{n-1} \xi_k (f) = \pi_\mu \mu(f) + \pi_\lambda \lambda(f) =
\E \xi_0(f) ,\qquad{\mathrm{a.s}}.
\end{equation}
Therefore, using Theorem~\ref{thm:prop1} and then \eqref{eq:markov-chain-ergodic-theorem-nova}, we have that
\begin{align*}
\lim_{n\to\infty}\frac{1}{n}\sum\nolimits_{k=0}^{n-1}f\circ X_k & = \E \xi_0 (f) \\
& = \mu(f) \Prob[\xi_0=\mu]+\lambda(f)\Prob[\xi_0=\lambda] \\
& = \left( f(1)\mu(1) + f(-1)\mu(-1)\right)\pi_\mu
+ \left( f(1)\lambda(1) + f(-1)\lambda(-1)\right)\pi_\lambda \\
& = f(1)\big(\mu(1)\pi_\mu +\lambda(1)\pi_\lambda\big) + f(-1)\big(\mu(-1)\pi_\mu +\lambda(-1)\pi_\lambda\big)
\end{align*}
holds almost surely.
In particular this is true with $f = \I_{\lbrace1\rbrace}$; thus, even though the `ups and downs' of $\bs X$ are governed by a law which can be rather complicated (as one suspects by inspecting equation \eqref{eq:markov-switching}), we can still estimate the overall (unconditional) probability of, say, the expansion regime by computing the proportion of ups in a sample $(X_0,\dots,X_n)$:
\[
\lim_{n\to\infty}\frac{1}{n}\sum_{k=0}^{n-1} \I_{\lbrace1\rbrace}(X_k) = \mu(1)\pi_\mu +\lambda(1)\pi_\lambda.
\]
\end{example}

\begin{example}
Suppose $\bs\vartheta = (\vartheta_0,\vartheta_1,\dots)$ is a submartingale, with $\mathrm{range}(\vartheta_n)\subseteq[0,1]$ for all $n$. By the Martingale Convergence Theorem, there exists a random variable $\vartheta_\infty$ such that $\lim \vartheta_n = \vartheta_\infty$ almost surely (thus, we can assume without loss of generality that $0\leq\vartheta_\infty\leq1$). Furthermore, let $S\coloneqq\{0,1\}$ and, for $n\in\N$, let $\xi_n$ denote the random probability measure on $S$ defined via $\xi_n(\{1\}) = \vartheta_n$, and $\xi_n(\{0\}) = 1-\vartheta_n$. We have $\xi_n(\I_{\lbrace1\rbrace}) = \vartheta_n \to\vartheta_\infty$ a.s. Assume further that $\bs\xi = (\xi_0, \xi_1,\dots)$ is a product disintegration of a sequence $\bs X = (X_0, X_1,\dots)$ of random variables with values in $S$. Using Theorem~\ref{thm:prop1} we have
\[ \lim_{n \to \infty} \frac{1}{n}\sum_{i=0}^{n-1} \I_{\lbrace1\rbrace}(X_i) =
\lim_{n \to \infty} \frac{1}{n}\sum_{i=0}^{n-1} \xi_i(\I_{\lbrace1\rbrace})
= \lim_{n \to \infty} \frac{1}{n}\sum_{i=0}^{n-1}\vartheta_i=\vartheta_\infty \quad \text{a.s.}
\]
which means, the proportion of 1's in $(X_0,\dots,X_n)$ approaches $\vartheta_\infty$ with probability one.

To illustrate, let $(U_n\colon n\geq1)$ be a sequence of independent and identically distributed Uniform$[0,1]$ random variables. Let $\vartheta_0 = U_0/2$ and, for $n\geq1$, define $\vartheta_{n} \coloneqq \vartheta_{n-1} + 2^{-(n+1)}U_{n}$. Figure \ref{fig:submartingale} displays, in blue, a simulated sample path of the submartingale $(\vartheta_n\colon n\in\N)$ up to $n=20$. The $\circ$'s represent the successive outcomes of the coin throws (where the probability of `heads' in the $n$th throw is $\vartheta_n$). In purple are displayed the sample path of the means $(n^{-1}S_n\colon n\in \N)$, where $S_n$ is the partial sum $\sum_{i=0}^{n-1}X_n$. In this model, even if we only observe the outcomes of the coin throws, we can still estimate the value of $\vartheta_\infty$: all we need to do is to compute the proportion of heads in $X_0,\dots,X_n$, with $n$ large.
\begin{figure}[h!]
\begin{center}
\includegraphics[scale=.9]{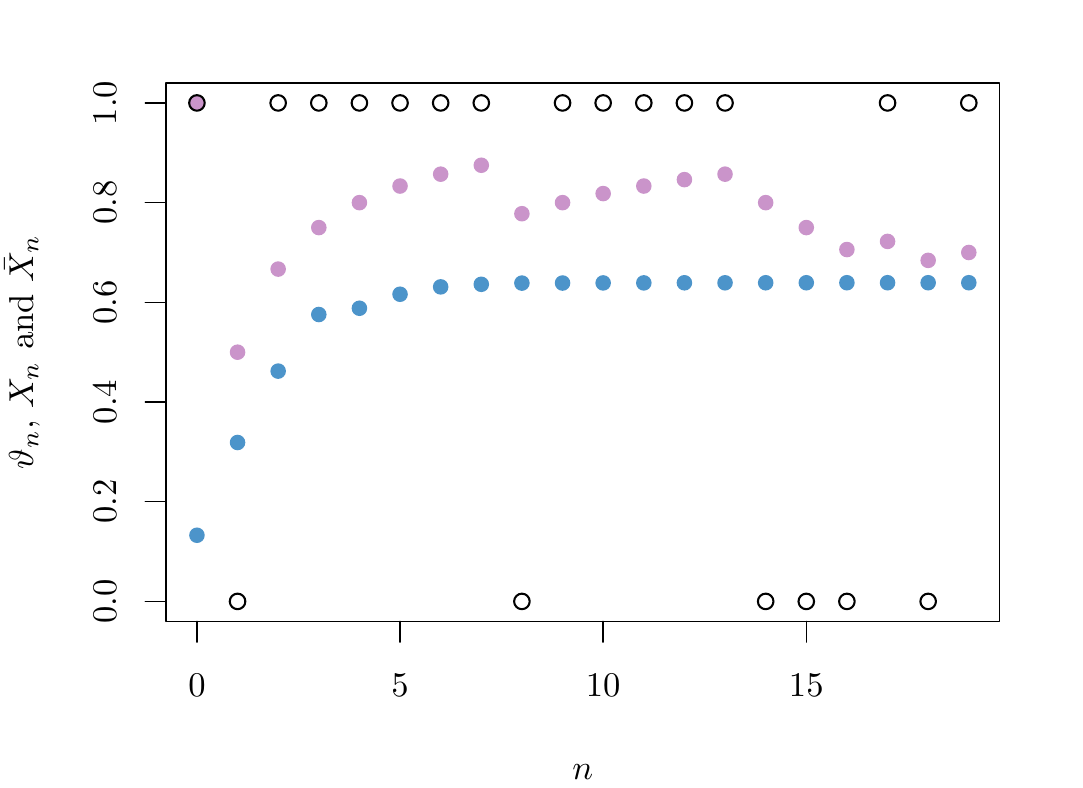}
\caption{A sample path of the submartingale $(\vartheta_n)$.}\label{fig:submartingale}
\end{center}
\end{figure}
\end{example}

The following lemma gives us a characterization of product disintegrations. It allows us to work with a requirement seemingly weaker than the almost sure validity of eq.\ \eqref{eq:weaker-product-measure-disintegration} uniformly for all integer $n\ge1$ and Borel sets $A_0,\dots,A_n$; in practice, this lemma is merely a ``rearrangement of quantifiers''. It will be used in the next example, the result itself as well as the reasoning used in its proof.

\begin{lemma}\label{thm:product-measure-disintegration}
Let $\bs{X} \coloneqq \left(X_0,X_1,\dots\right)$ be a sequence of random variables taking values in a compact metric space $S$, and let $\bs{\xi} = (\xi_0,\xi_1,\dots)$ be a sequence of random probability measures on $S$. Then $\bs\xi$ is a product disintegration of $\bs X$ if and only if for each $n$ and each $(n+1)$-tuple $A_0,\dots,A_n$ of measurable subsets of $S$, the equality \eqref{eq:weaker-product-measure-disintegration} holds almost surely.
\end{lemma}

\begin{proof}
The `only if' part of the statement is trivial. For the `if' part, let $\mathscr{S}^\N$ denote the product $\sigma$-field on $S^\N$. By Lemma 1.2 in \cite{kallenberg2002foundations}, $\mathscr{S}^\N$ coincides with the Borel $\sigma$-field corresponding to the product topology on $S^\N$, and therefore $S^\N$ is a Borel space. By Theorem 6.3 in \cite{kallenberg2002foundations}, there exists an event $\Omega^*\subseteq\Omega$ with $\Prob(\Omega^*)=1$ such that
\(
\bs{A}\mapsto \Prob(\bs{X}\in\bs{A}\given\bs{\xi})_\omega
\)
is a probability measure on $\mathscr{S}^\N$ for each $\omega\in\Omega^*$.

Now let $\mathfrak{C}\coloneqq\{\bs{A}_k\colon\,k\in\N\}$ be a countable collection of sets of the form $\bs{A}_k = B_0^k\times\cdots\times B_{n(k)}^k\times S\times\cdots$ which generates $\mathscr{S}^\N$. By assumption, for each $k$ there is an event $\Omega_k\subseteq \Omega$ with $\Prob(\Omega_k)=1$ such that
\(
\Prob(\bs{X}\in\bs{A}_k\given\bs{\xi})_\omega = \xi_0^\omega(B_0^k)\cdots\xi_{n(k)}^\omega(B_{n(k)}^k)
\)
holds for $\omega\in\Omega_k$. Thus, for $\omega\in \Omega'\coloneqq \left(\bigcap_{k=0}^\infty \bs{A}_k\right)\cap \Omega^*$, with $\Prob(\Omega')=1$, the probability measures $\Prob(\bs{X}\in\cdot\given\bs{\xi})_\omega$ and $\prod_{n=0}^\infty\xi_n^\omega$ agree on a $\pi$-system which generates $\mathscr{S}^\N$, and therefore they agree on $\mathscr{S}^\N$. This establishes the stated result.
\end{proof}

We now present an example where product disintegrations can be applied to a problem in finance.

\begin{example}[Stochastic Volatility models]
This example demonstrates that a certain class of stochastic volatility models can be accommodated into our framework of product disintegrations. Stochastic volatility models are widely used in the financial econometrics literature, as they provide a parsimonious approach for describing the volatility dynamics of a financial asset's return --- see \cite{shephard2009overview} and \cite{davis2009svmodels} and references therein for an overview. A basic specification of the model\footnote{Which can be relaxed by putting $g\circ H_t$ in place of $\ee^{H_t/2}$, and allowing $\bs H$ to evolve according to more flexible dynamics.} is as follows: let $\bs Z = (Z_t\colon t\in\Z)$ and $\bs W = (W_t\colon t\in \Z)$ be centered iid sequences, independent from one another, and define $\bs X$ and $\bs H$ via the stochastic difference equations
\[
X_t = \ee^{H_t/2}Z_t,\quad t\in\N,
\quad\mbox{and}\quad
H_t = \alpha + \beta H_{t-1} + W_t,\quad t\geq 1,
\]
where $\alpha$ and $\beta$ are real constants and where $H_0$ follows some prescribed distribution. The random variable $X_t$ is interpreted as the return (log-price variation) on a given financial asset at date $t$, and the $H_t$'s are latent (i.e, unobservable) random variables that conduct the volatility of the process $\bs X$. Usually this process is modelled with Gaussian innovations, that is, with $W_t$ and $Z_t$ normally distributed for all $t$. In this case the random variables $X_t$ are supported on the whole real line, so we need to consider other distributions for $\bs Z$ and $\bs W$ if we want to ensure that the $X_t$'s are compactly supported.

Our objective is to show how Theorem~\ref{thm:prop1} can be used to estimate certain functionals of the latent volatility process $\bs H$ in terms of the observed return process $\bs X$. To begin with, notice that if $|\beta|<1$ and if $H_0$ is defined via the series $H_0 \coloneqq {(1-\beta)}^{-1} \alpha+\sum_{k=0}^{\infty} \beta^k W_{-k}$, then $\bs H$ (and $\bs X$) is strictly stationary and ergodic, in which case we have that
\begin{equation}
\lim_{n\to\infty} \frac1n\sum_{t=0}^{n-1}
g\circ H_t = \E g\circ H_0
\end{equation}
almost surely, for any $\Prob_{H_0}$-integrable $g\colon S_H\to\R$, where we write $S_H \coloneqq \mathrm{supp}\, H_0$. Also, notice that, by construction, we have for all $n$, all measurable $A_0,\dots,A_n\subseteq S \coloneqq \mathrm{supp}\,X_0$ and all $\bs h = (h_0,h_1,\dots)\in S_H^{\N}$,
\begin{align}
\nonumber\Prob\big(X_0 \in A_0,\dots,X_n\in A_n\,|\,\bs H = \bs h\big) &=
\Prob\big(\ee^{H_0/2}Z_0\in A_0,\dots, \ee^{H_n/2}Z_n\in A_n\given \bs H = \bs h\big)\\
\nonumber&\overset{(*)}=\Prob\big(\ee^{h_0/2}Z_0\in A_0,\dots, \ee^{h_n/2}Z_n\in A_n\given \bs H = \bs h\big)\\
\nonumber&\overset{(**)}=\Prob\big(\ee^{h_0/2}Z_0\in A_0,\dots, \ee^{h_n/2}Z_n\in A_n\big)\\
\nonumber&= \prod_{t=0}^n
\Prob\big(\ee^{h_t/2}Z_t\in A_t\big)\\
&\overset{(***)}= \prod_{t=0}^n
\Prob\big(X_t\in A_t\,|\,\bs H = \bs h\big).\label{eq:product-disintegraion-H}
\end{align}
Where $(*)$ is yielded by the substitution principle, $(**)$ follows from the fact that $\bs Z$ and $\bs H$ are independent (as $\bs H$ only depends on $\bs W$), and $(***)$ is just a matter of repeating the previous steps. A reasoning similar to the one used in the proof of Lemma~\ref{thm:product-measure-disintegration} then tells us that $\Prob(\bs X\in\cdot\,|\,\bs H)_\omega$ is a product measure on $S^{\N}$ for almost all $\omega$. Also, notice that in particular we have that $\Prob(X_t\in A\given \bs H = \bs h) = \Prob(\ee^{h_t/2}Z_t\in A)$ for all $t$. In fact, let $\varphi\colon S_H\to M_1(S)$ be defined via $\varphi(h,A):= \Prob(\ee^{h/2}Z_0\in A)$, for $h\in S_H$ and measurable $A\subseteq S$, where we write $\varphi(h,A)$ in place of $\varphi(h)(A)$ for convenience. Since the $Z_t$'s are identically distributed, we have in particular that $\varphi(h,A) = \Prob(\ee^{h/2}Z_t\in A)$ for all $t$. The preceding derivations now allow us to conclude that
\begin{equation}\label{eq:product-disintegraion-H-2}
\varphi(H_t(\omega),A) = \Prob(X_t\in A\given H_t)_\omega = \Prob(X_t\in A\given \bs H)_\omega.
\end{equation}

We are now in place to introduce a product disintegration of $\bs X$, by defining $\xi_t^\omega(A):= \Prob(X_t\in A\given \bs H)_\omega$ for measurable $A\subseteq S$, $t\in\Z$ and $\omega\in \Omega$. To see that $\bs{\xi} = (\xi_0,\xi_1,\dots)$ is indeed a product disintegration of $\bs X$, first notice that $\xi_0(A_0)\cdots\xi_n(A_n)$ is $\sigma(\bs \xi)$-measurable for every $n$ and every $(n+1)$-tuple $A_0,\dots,A_n$ of measurable subsets of $S$. Moreover, defining $\psi\colon S_H^\N\to M_1(S)^\N$ via $\psi(h_0,h_1,\dots) = (\varphi(h_0), \varphi(h_1),\dots)$, we obtain, by equations~\eqref{eq:product-disintegraion-H} and \eqref{eq:product-disintegraion-H-2},
\begin{align*}
\E\big(\xi_0(A_0)\cdots\xi_n(A_n)\I_{[\bs\xi\in\bs B]}\big) &= \E\big(\varphi(H_0,A_0)\cdots\varphi(H_n,A_n)\I_{[\bs H \in \psi^{-1}(\bs B)]}\big)\\
&= \Prob\big(X_0\in A_0,\dots, X_n\in A_n, \bs H\in \psi^{-1}(\bs B)\big)\\
&= \Prob\big(X_0\in A_0,\dots, X_n\in A_n, \bs\xi\in \bs B\big),
\end{align*}
whence $\Prob(X_0\in A_0,\dots,X_n\in A_n\,|\,\bs\xi) = \xi_0(A_0)\cdots\xi_n(A_n)$, and then Lemma~\ref{thm:product-measure-disintegration} tells us that $\bs\xi$ is --- voilà --- a product disintegration of $\bs X$.

Now, since $\varphi$ is continuous and one-to-one, we have that $\varphi$ is a homeomorphism from $S_H$ onto its range whenever $S_H$ is compact (in particular, $\mathrm{range}(\varphi)$ is compact, hence measurable, in $M_1(S)$). Also, as $\xi_t = \varphi\circ H_t$ for all $t$, we have that $H_t = \varphi^{-1}\circ \xi_t$ is well defined. Suppose now that $f\colon S\to\R$ is a given continuous function. We have \[\xi_t^\omega(f) = \int_S f(x)\,\xi_t^\omega(\dd x) = \int_S f(x)\,\varphi(H_t(\omega),\dd x) =: g(H_t(\omega))\] and, as $\bs H$ is ergodic, it holds that
\(
\lim_{n\to\infty} n^{-1} \sum_{t=0}^{n-1} \xi_t(f) = \E g\circ H_0,
\)
where we know that the expectation is well defined, as
\(
\E |g\circ H_0| \leq \E \left(\int_S |f(x)|\,\xi_0(\dd x) \right)< \infty,
\)
with the expected value given by $\E g\circ H_0 = \int_S f(x)\,\Prob_{X_0}(\dd x)$.
We can now apply Theorem~\ref{thm:prop1} to see that
\[
\E g\circ H_0 = \lim_{n\to\infty} n^{-1} \sum_{t=0}^{n-1} f\circ X_t.
\]
The conclusion is that, for suitable $g$ of the form $g(h) = \int_S f(x)\,\varphi(h,\dd x)$, we can estimate $\E g\circ H_0$ by the data $(X_0, X_1, \dots, X_n)$ as long as $n$ is large enough, even if we cannot observe $\bs H$. Of course, this follows from ergodicity of $\bs X$, but it is interesting anyway to arrive at this result from an alternate perspective; moreover, one can use Hoeffding type inequalities as in Example~\ref{example:concentration} to easily derive a rate of convergence for sample means of $\bs X$ based on the rate of convergence of sample means of $\bs H$.
\end{example}

\section{Acknowledgements}

The author Luísa Borsato was supported by grants
2018/21067-0 and 2019/08349-9, São Paulo Research Foundation (FAPESP).
%The author Eduardo Horta wishes to thank MCTIC/CNPq (process number 438642/2018-0) for financial support.

\end{document}